\documentclass[12pt]{amsart}

\numberwithin{equation}{section}

\usepackage{amsmath,amsthm,amsfonts,eucal,}
\usepackage{xypic}
\usepackage{graphicx}
\usepackage{amssymb}

\def\cb{{\mathcal B}}

\def\cf{{\mathcal F}}

\def\ch{{\mathcal H}}

\def\cs{{\mathcal S}}

\def\ga{{\mathfrak A}} 

\def\gc{{\mathfrak C}}

\def\bc{{\mathbb C}}

\def\bl{{\mathbb L}}
\def\bj{{\mathbb J}}

\def\bm{{\mathbb M}}
\def\bn{{\mathbb N}}
\def\bo{{\mathbb O}}
\def\bp{{\mathbb P}}

\def\bz{{\mathbb Z}}

\def\a{\alpha}

\def\g{\gamma}  
\def\d{\delta}

\def\p{\pi}

\def\r{\rho}

\def\f{\varphi}  
  
\def\om{\omega} \def\Om{\Omega}

\newtheorem{thm}{Theorem}[section]

\newtheorem{lem}[thm]{Lemma}
\newtheorem{cor}[thm]{Corollary}
\newtheorem{prop}[thm]{Proposition}

\theoremstyle{definition}
\newtheorem{rem}[thm]{Remark}

\def\carf{\mathop{\rm CAR}}

\begin{document}
\title[On distributional symmetries on the CAR algebra]
{On distributional symmetries on the CAR algebra}

\author{Vitonofrio Crismale}
\address{Vitonofrio Crismale\\
Dipartimento di Matematica\\
Universit\`{a} degli studi di Bari\\
Via E. Orabona, 4, 70125 Bari, Italy}
\email{\texttt{vitonofrio.crismale@uniba.it}}

\author{Simone Del Vecchio}
\address{Simone Del Vecchio\\
Dipartimento di Matematica\\
Universit\`{a} degli studi di Bari\\
Via E. Orabona, 4, 70125 Bari, Italy}
\email{\texttt{simone.delvecchio@uniba.it}}

\author{Stefano Rossi}
\address{Stefano Rossi\\
Dipartimento di Matematica\\
Universit\`{a} degli studi di Bari\\
Via E. Orabona, 4, 70125 Bari, Italy}
\email{\texttt{stefano.rossi@uniba.it}}
%\date{\today}

\begin{abstract}

Spreadable, exchangeable, and rotatable states on the CAR algebra are shown to be
the same.

.

\end{abstract}

\maketitle

\section{Introduction}

A (discrete) stochastic process enjoys a distributional symmetry when its finite-dimensional joint distributions
are invariant under the action of remarkable groups or semigroups
of measurable transformations. Among these, a notable role is played
by the group of translations, the group of finite permutations, the group of rotations, and the semigroup
of increasing monotone maps. The corresponding distributional symmetries are known as
stationarity, exchangeability, rotatability and spreadability, respectively.
Exchangeability is settled by de Finetti's theorem that a
sequence of random variables is exchangeable if and only if its joint distribution is a mixture of
measures each of which is the joint distribution of an independent and identically distributed process, see \cite{dF,HS}. Rotatability is a stronger symmetry since permutations can be seen as orthogonal transformations.
As a consequence, for rotatable sequences  the common distribution of the
independent and identically distributed processes above cannot be arbitrary. In fact,
 it is forced to be Gaussian, which is the content of
Freedman's theorem, \cite{Fr,Ka}.
Lastly, spreadability should in principle be a weaker symmetry than exchangeability. However, it turns out to be the same as
exchangeability in light of a classical result of Ryll-Nardzewski, \cite{R}.\\
Distributional symmetries for quantum stochastic processes can be considered as well and have recently  been focused on in a number of
papers, see {\it e.g.} \cite{CrFid, CrFid2, CFG2, CDR, KKW, CDMR}. In this paper, we mainly deal with processes based on the CAR algebra $\carf(\bz)$.
These are sequences of
quantum random variables with sample algebra
$\bm_2(\bc)$ and which are subject to the canonical anticommutation relations. As explained in {\it e.g.} \cite{CrFid2, CRcar},
processes of this type are in a one-to-one correspondence with states on $\carf(\bz)$. Furthermore, the distributional
symmetry of such a process  amounts to the invariance of the corresponding state.\\
The set $\cs^{\bp_\bz}(\carf(\bz))$  of exchangeable states of the CAR algebra  has been addressed in \cite{CFCMP}, where it is shown to be
a Choquet simplex whose extremes are Araki-Moriya product states \cite{AM1} of a single even state of $\bm_2(\bc)$.

In this paper, we show that on the CAR algebra spreadable states are the same as exchangeable states, Theorem \ref{main},  thus establishing a fermionic version of the Ryll-Nardzewski theorem \cite{R}.
This fixes a mistake contained in (both statement and proof of) \cite[Theorem 4.2]{CRcar}  from previous work by two of the authors of the present paper.
Moreover, in Proposition \ref{face} we show that the set of spreadable/exchangeable states is somewhat small in the larger set of all stationary states being a proper face. 
Theorem  \ref{main} allows us to come to an accomplished characterization of symmetric states, which is the fermionic version of the
so-called extended de Finetti theorem. Precisely, combining our result with those in \cite{CFCMP, CrFid}, we have that
spreadability and exchangeablity for processes on the CAR algebra are both the same as conditional independence
and identical distribution with respect to the tail algebra. Interestingly, the techniques we employ to prove our result
are not limited to the CAR algebra. In fact, they can be used to handle tensor products of
$C^*$-algebras, thus including the formalism established by St\o rmer in \cite{Sto}. Among other 
things, applying the general result to commutative sample $C^*$-algebras  yields a proof of the classical Ryll-Nardzewski theorem for essentially bounded
random variables.

The strategy to arrive at our results has an interest in its own right not least because
non-commutative ergodic theory for automorphic actions of \emph{groups} on
$C^*$-algebras does not apply as spreadability is implemented by the action of the \emph{semigroup}
$\bj_\bz$ of all monotone increasing functions of $\bz$ to itself with cofinite range.
This obstacle, however, can be circumvented by using a geometric construction naturally
arising from the symmetry under consideration which allows one to
realize  the set of all spreadable states on the CAR algebra  as the set of all states on a larger CAR algebra that are invariant
under the action of a group, Theorem \ref{affine}.
In more detail, the larger algebra is $\carf(\bz\big[\frac{1}{2}\big])$, {\it i.e.} the CAR
algebra on the set $\bz\big[\frac{1}{2}\big]$ of dyadic numbers, obtained as inductive limit of
the increasing sequence $\{\carf\big(\frac{\bz}{2^n}\big): n\in\bn\}$ of CAR algebras.
Notably, this construction features the property that any spreadable state on $\carf\big(\frac{\bz}{2^n}\big)$
 uniquely extends to a spreadable state on $\carf\big(\frac{\bz}{2^{n+1}}\big)$, and the
inductive limit state one ends up with on $\carf(\bz\big[\frac{1}{2}\big])$ is invariant under the action
of a group $G$ of  continuous piecewise linear bijections of $\bz\big[\frac{1}{2}\big]$.
The group $G$ is in turn an inductive limit of a sequence of groups $G_n$.
%each of which is isomorphicto Thompson's group $F$.

$\carf(\bz)$ is also acted on by orthogonal transformations in a natural way through
Bogolubov automorphisms. This enables us to consider rotatability for states on the CAR
algebra as well.  Contrary to the classical case, rotatable states, $\cs^{\bo_\bz}(\carf(\bz))$, are exactly the same as exchangeable states,
Proposition \ref{rot}. The overall picture of invariant states on $\carf(\bz)$ is then the following
$$\cs^{\bo_\bz}(\carf(\bz))=\cs^{\bp_\bz}(\carf(\bz))=\cs^{\bj_\bz}(\carf(\bz))\subsetneq\cs^\bz(\carf(\bz)),$$
where $\cs^\bz(\carf(\bz))$ denotes the stationary states and the last strict inclusion is shown in{ \it e.g.} \cite{CRcar}.

Finally, the so-called self-adjoint subalgebra $\mathfrak{C}(\bz)$ of $\carf(\bz)$
is also addressed, since its generators are the Fermi analogues of Bernoulli
random variables. As proved in \cite{CRcar}, the vacuum state is the only exchangeable state on
 $\mathfrak{C}(\bz)$; therefore,  it is also its only rotatable state. Moreover, we prove in Corollary \ref{selfadjoint} that the vacuum state is the only spreadable state on $\mathfrak{C}(\bz)$.
The picture of distributional symmetries on
 $\mathfrak{C}(\bz)$ is the same as what we saw above for the CAR algebra. 
Phrased differently, the sets of spreadable, symmetric, and rotatable states on $\mathfrak{C}(\bz)$ 
all coincide with the singleton made up of the vacuum state, whereas there exist stationary states other than it, {\it cf.} \cite {CRcar}, Proposition 4.1.

\section{Preliminaries}
By a realization of a {\it quantum stochastic process} labelled by the index set $J$ we mean
 a quadruple
$\big(\ga,\ch,\{\iota_j\}_{j\in J},\xi\big)$, where $\ga$ is a (unital) $C^{*}$-algebra, referred to as the sample
algebra of the process, $\ch$ is a Hilbert space,
whose inner product is denoted by $\langle\cdot,\cdot \rangle$,
the maps $\iota_j$ are (unital) $*$-representation of $\ga$ on $\ch$, and
$\xi\in\ch$ is a unit vector, which is cyclic for  the von Neumann algebra
$\bigvee_{j\in J}\iota_j(\ga)''$.\\
The assignment of such a quadruple amounts to a state $\varphi$ on
the free product $C^*$-algebra $\ast_{J} \ga$. For free products of $C^*$-algebras the reader is referred to
\cite{A}. For completeness, we next quickly recall how the correspondence between stochastic
processes and states works. If
one starts
with a stochastic process, then a state
$\varphi$ on the free product $\ast_{J} \ga$ can be defined by setting
\begin{equation}\label{corresp}
\varphi(i_{j_1}(a_1)i_{j_2}(a_2)\cdots i_{j_n}(a_n)):=\langle(\iota_{j_1}(a_1)\iota_{j_2}(a_2)\cdots\iota_{j_n}(a_n)\xi, \xi\rangle
\end{equation}
for all $n\in\bn$, $j_1\neq j_2\neq\cdots\neq j_n\in J$ and
$a_1, a_2, \ldots, a_n\in\ga$, where $i_j:\ga\rightarrow \ast_{J}\ga$ is
the $j$-th embedding of $\ga$ into $\ast_{J} \ga$.
The values of $\varphi$ on monomials of the type above are
often referred to as finite-dimensional distributions of the process itself.\\
Conversely, all states on the free product $\ast_{J} \ga$ arise in this way,
see  \cite[Theorem 3.4]{CrFid}.
Phrased differently, starting now with a state $\f\in\cs\big(\ast_{J}\ga\big)$,  a stochastic process can be defined by using  the GNS representation $(\p_\f, \ch_\f, \xi_\f)$ of
$\f$. Indeed, for every $j\in J$ we can set $\iota_j(a):=\pi_\f(i_j(a))$, $a\in\ga$, so as to get the quadruple
$\big(\ga,\ch_\f,\{\iota_j\}_{j\in J},\xi_\f\big)$. \\

The {\it Canonical Anticommutation
Relations} (CAR for short) algebra over  $\bz$ is the universal unital
$C^{*}$-algebra $\carf(\bz)$,  with unit $I$, generated by
the set $\{a_j, a^{\dagger}_j: j\in \bz\}$ ({\it i.e.} the Fermi
annihilators and creators respectively),  satisfying the relations
\begin{equation}\label{stara}
(a_{j})^{*}=a^{\dagger}_{j}\,,\,\,\{a^{\dagger}_{j},a_{k}\}=\d_{j, k}I\,,\,\,
\{a_{j},a_{k}\}=\{a^{\dagger}_{j},a^{\dagger}_{k}\}=0\,,\,\,j,k\in \bz\,.
\end{equation}
where $\{\cdot, \cdot\}$ is the anticommutator and  $\d_{j, k}$ is the Kronecker
symbol.\\
The CAR algebra is the completion of the subalgebra of so-called
localized elements. More precisely,

\begin{equation*}
\carf(\bz)=\overline{\carf{}_0(\bz)}\,,
\end{equation*}
where
$$
\carf{}_0(\bz):=\bigcup\{\carf(F): F\subset \bz\,\text{finite}\,\}
$$
and
$\carf(F)$ is the $C^*$-subalgebra
generated by the finite set $\{a_j, a^\dag_j: j\in F\}$.
We will sometimes say that an element $a$ of $\carf(\bz)$ is localized
in $F$ if  $a\in\carf(F)$.\\
At this point, we also recall that CAR algebras can be associated functorially
to Hilbert spaces as well, the relation of the two constructions being that the CAR algebra indexed 
by a set $S$ and the CAR algebra indexed by a space $\ch$ are isomorphic when the cardinality of an orthonormal basis of $\ch$ is $|S|$. The advantage of working with the Hilbert space approach is that the  annihilators 
are defined in correspondence of any vector $f$ in $\ch$ in such a way that
$f\rightarrow a(f)$ is linear w.r.t. $f$, and, if $\{e_i: i\in S\}$ is an orthonormal basis, then $a(e_i)$ is just what we denoted by $a_i$ for all $i$ in $S$.\\
The parity automorphism $\theta$ acting on the generators as
$$
\theta(a_{j})=-a_{j}\,,\,\,\theta(a^{\dagger}_{j})=-a^{\dagger}_{j}\,,\quad
j\in \bz\,,
$$
makes the CAR algebra into a $\bz_2$-graded algebra.
Clearly, $\carf(\bz)$ decomposes as $\carf(\bz)=\carf(\bz)_{+} \oplus\carf(\bz )_{-}$, where
\begin{align*}
&\carf(\bz)_{+}:=\{a\in\carf(\bz) \ | \ \theta(a)=a\}\,,\\
&\carf(\bz)_{-}:=\{a\in\carf(\bz) \ | \ \theta(a)=-a\}\,.
\end{align*}
Elements  in $\carf(\bz)_+$ and in $\carf(\bz)_-$ are called
{\it even} and {\it odd}, respectively.\\
A state $\varphi$ on $\carf(\bz)$ is said to be {\it even}
if $\varphi\circ\theta=\varphi$, which is the same as $\f\lceil_{\carf(\bz)_{-}}=0$, where
$\f\lceil_{\carf(\bz)_{-}}$ is the restriction to $\carf(\bz)_{-}$ of $\varphi.$
\\
The CAR algebra can be presented in several equivalent ways. Below we recollect two
presentations that are particularly suited to the purposes of the paper.
First, the CAR algebra is  isomorphic
with the infinite graded tensor product of  the $\bz_2$-graded $C^*$-algebra
$(\bm_2(\bc), {\rm ad}(U))$ with itself, where ${\rm ad}(U)(\cdot):=U(\cdot) U$
is the grading induced by
the self-adjoint unitary (Pauli) matrix $U:=  \left(
\begin{array}{rr}
1 & 0 \\
0 &{ -1}
\end{array}
\right)$, see Example 3.2 of \cite{CRZbis}.\\
Second, the CAR algebra can also be seen as a quotient of the free product
$\ast_\bz \bm_2(\bc)$. Indeed, if we define $A:=  \left(
\begin{array}{rr}
0 & 1 \\
0 &0
\end{array}
\right)$, it is easy to see that  the quotient of $\ast_\bz \bm_2(\bc)$
modulo the relations $\{i_j(A^*), i_k(A)\}=\d_{j, k}I$
and $\{i_j(A),i_k(A)\}=\{i_j(A^*),i_k(A^*)\}=0$, for
all $j, k\in\bz$, is isomorphic with the CAR algebra by \eqref{stara}. It is this very presentation
of the CAR algebra that allows one to interpret its states as particular stochastic processes defined on the sample
algebra $\bm_2(\bc)$. \\
By a slight abuse of notation, for every $k\in\bz$, we continue to denote by
$i_k: \bm_2(\bc)\rightarrow \carf({\bz})$ the composition  $\pi\circ i_k$, where
$\pi:\ast_\bz \bm_2(\bc)\rightarrow \carf(\bz) $ is the quotient map.\\
For completeness' sake, we  rather quickly recall  that $\carf(\bz)$ has a distinguished (faithful) irreducible
representation on the Fermi Fock space $\cf_-(\ell^2(\bz))$, where for every
$j\in\bz$, the operator $a_j^\dag$ (or $a_j$) acts as
the Fermi creator  (or annihilator) of a particle in the state $e_j$,
$\{e_j:j\in\bz\}$ being the canonical orthonormal basis of
$\ell_2(\bz)$. For an exhaustive account of the Fermi Fock space the reader is
referred to Chapter 5.2 of \cite{BR2}.
The vector state associated with the Fock vacuum vector $\Om\in\cf_-(\ell_2(\bz))$ ({\it i.e.} the one corresponding to the state with no particles at all) is called the vacuum state and denoted by
$\om_\Om$.\\
Along with $\carf(\bz)$, we will also need to consider the
$C^*$-subalgebra $\gc(\bz)\subset \carf(\bz)$ generated by
the so-called position operators $x_i:=a_i+a_i^\dag$, $i\in\bz$.
In the literature, the $C^*$-algebra $\gc(\bz)$ is often referred to
as the self-adjoint subalgebra of the CAR algebra, and we will stick to this terminology
as well.\\

Comparing the distributional symmetries involved in the present paper requires considering some algebraic structures on $\bz$.\\
%We first consider the group generated by the one-step shift $\t(i):=i+1$ of the integers $\bz$, which is canonically identified with $\bz$ itself.\\
We denote by $\bp_\bz$ the group of finite permutations of the set $\bz$.
Its elements are  bijective maps of $\bz$ which only moves finitely many integers.  The group operation
is given by the map composition.\\
We denote by $\bl_\bz$ the unital semigroup of all strictly increasing maps of $\bz$ to itself.
As is known from Classical Probability, the importance of  the semigroup $\bl_\bz$ lies in the role it plays
in relation with spreadability.
However, this distributional symmetry is equivalently implemented by the action of a smaller semigroup,
$\bj_\bz$, which was introduced in \cite{CFG2}. This is by definition the set of all strictly increasing
maps $f$ of $\bz$ to itself whose range is cofinite, that is $|\bz\setminus f(\bz)|<\infty$, where for any
set $E$, $|E|$ denotes the cardinality of $E$.
We next describe a useful set of generators of $\bj_\bz$. For every fixed $h\in\bz$, define the $h$-{right hand-side partial shift} as the element $\theta_h$ of $\bj_\bz$ given by
$$
\theta_h(k):=\left\{\begin{array}{ll}
                      k & \text{if}\,\, k<h\,, \\
                      k+1 & \text{if}\,\, k\geq h\,.
                    \end{array}
                    \right.
$$
In \cite{CFG, CFG2}  the set $\{\theta_h, \tau, \tau^{-1}: h\in\bz\}$ has been shown to generate
$\bj_\bz$ as a unital semigroup, where $\tau$ is given by $\tau(k)=k+1$, $k\in\bz$.\\
The semigroup $\bj_\bz$ is known to be left but not right amenable, see \cite{CRcar}.
Despite failing to be right amenable, $\bj_\bz$ does have a right F{\o}lner sequence. Precisely,
the sequence of finite sets $\{F_n: n\in\bn\}\subset\bj_\bz$, with
$$
F_n:=\left\{\theta_{-n}^{h_{-n}}\theta_{-n+1}^{h_{-n+1}}\cdots\theta_0^{h_0}\cdots\theta_{n-1}^{h_{n-1}} \theta_n^{h_n}\tau^l: \,\, \sum_{i=-n}^n h_i\leq n^2,\,\, -n\leq l\leq n\right\}
$$
is proved to be a right F{\o}lner sequence in
\cite[Proposition 3.1]{CRcar}. Explicitly, this means that for any fixed
$h\in\bj_\bz$, one has
$$\lim_{n\rightarrow\infty}\frac{|F_n\Delta F_nh|}{|F_n|}=0\,,$$
where $\Delta$ denotes the symmetric difference.\\

By universality of the CAR algebra, the natural action on $\bz$
of both $\bp_\bz$ and $\bj_\bz$ can be lifted to the CAR algebra by making them act directly on the indices of the generators $a_i$. More precisely, if $h$ is in either $\bp_\bz$ or $\bj_\bz$, there exists a unique endomorphism $\alpha_h$ of
$\carf(\bz)$ uniquely determined by
$$\alpha_h(a_i)=a_{h(i)}\,, i\in\bz\, .$$
Accordingly, one can consider the invariant  states for these actions. In light of Equality \eqref{corresp}, the associated
processes will enjoy the relative distributional symmetry. In particular, invariant states for the action
of $\bp_\bz$ are known as \emph{symmetric} or \emph{exchangeable} states and the corresponding processes are often referred to as exchangeable processes. We will denote by $\cs^{\bp_\bz}(\carf(\bz))$  the compact convex set of all symmetric states
on the CAR algebra.\\		
Invariant states under the action of $\bj_\bz$ are known as \emph{spreadable} states and the corresponding processes are referred to as spreadable processes as well. We will denote by $\cs^{\bj_\bz}(\carf(\bz))$ the compact convex set of all spreadable states on the CAR algebra.\\
Stationary states, which we will denote by $\cs^\bz(\carf(\bz))$, are those invariant under the action of 
the automorphism $\alpha_\tau$, where $\tau$ is the right shift map, that is $\tau(k)=k+1$, $k\in\bz$.
Note that by definition there holds the inclusion $\cs^{\bj_\bz}(\carf(\bz))\subseteq \cs^\bz(\carf(\bz))$.\\

\noindent
Denote by $\bo_\bz$ the group of (infinite) orthogonal matrices $O:=[O_{i,j}]_{i,j\in\bz}$ such that there are only
finitely many entries $O_{i, j}$ not equal to $\d_{i,j}$, the Kronecker symbol.
We briefly recall how $\mathbb{O}_\bz$ acts on ${\rm CAR}(\bz)$ in a natural way.
Given a matrix $O=[O_{i,j}]\in\mathbb{O}_\bz$, set
$$
\rho_O(a_i):=\sum_{k\in\bz} O_{k,i}a_k \quad i\in\bz\,.
$$
In order to verify that each $\rho_O$ defines an automorphism of
${\rm CAR}(\bz)$, by universality we only need to check that $\{\rho_O(a_i): i\in\bz\}$
still satisfies the defining relations of the CAR algebra. One has
$\{\rho_O(a_i), \rho_O(a_j)\}=0$ and $\{\rho_O(a_i), \rho_O(a^\dag_j)\}=\delta_{i,j} I$
for any $i, j\in\bz$. The first equality is trivially verified. As for the second, we have:
\begin{align*}
\{\rho_O(a_i), \rho_O(a^\dag_j)\}&=\sum_{k, l\in\bz} O_{k,i} O_{l, j}\{a_k, a^\dag_l\}=
\sum_{k, l\in\bz} O_{k,i} O_{l, j}\delta_{k,l}I\\
&=\sum_{k\in\bz} O_{k,i} O_{k, j}I=\delta_{i,j}I
\end{align*}
where in the last equality we have exploited the orthogonality conditions. Finally, the equality
$\r_{O_1}\r_{O_2}=\r_{O_1O_2}$ can be immediately verified as well for any
$O_1, O_2\in\mathbb{O}_\bz$.\\
The automorphisms $\rho_O$ are known as Bogolubov automorphisms.
Actually, in the second-quantization formalism they are obtained  by functoriality in
a way that we next sketch for completeness' sake.
First, it is convenient to think of $\carf(\bz)$
as the concrete $C^*$-algebra $\carf(\ell_2(\bz))$.
Associated with any unitary $U$ acting on
$\ell_2(\bz)$ there is a unitary $\mathcal{F}(U)$ acting on the Fermi
Fock space $\mathcal{F}_-(\ell_2(\bz))$.
Now the inner automorphism ${\rm ad}(\mathcal{F}(U))$ restricts to an automorphism of
$\carf(\bz)$, see {\it e.g.} \cite{BR2}.
The automorphisms $\rho_O$ defined above are nothing but  the restriction of ${\rm ad}(\mathcal{F}(O))$
to the CAR algebra, where
by a minor abuse of notation $O$ denotes the unitary operator of $\ell_2(\bz)$ defined as
$$O e_i:= \sum_{k\in\bz} O_{k, i} e_k\,,\,\, i\in\bz\, ,$$
 $\{e_i: i\in\bz\}$ being the canonical orthonormal basis
of $\ell_2(\bz)$.\\
We will say that a state $\om$ on ${\rm CAR}(\bz)$ is
 \emph{rotatable} if it is  invariant
under the above action of $\mathbb{O}_\bz$, namely if
$\om\circ\rho_O=\om$ for every $O\in\mathbb{O}_\bz$.
The set of all rotatable states will be denoted by $\cs^{\bo_\bz}(\carf(\bz))$.
The terminology comes from Classical Probability in that the finite-dimensional
distributions of the process corresponding
to such a state will be invariant under orthogonal transformations
thanks to \eqref{corresp}.\\
At this point it is also worth noting that the states we are going to consider in this paper
are all automatically even because stationarity implies evenness, {\it cf.}
Example 5.2.21 in \cite{BR2}.\\

The actions of $\bp_\bz$ and $\bo_\bz$ on the CAR algebra are examples of a
$C^*$-dynamical system, by which we mean a triple
$(\ga, H, \beta)$, where $\ga$ is a $C^*$-algebra, $H$ is a group, and
$\beta: H\rightarrow {\rm Aut}(\ga)$ is a group homomorphism
from $H$ to ${\rm Aut}(\ga)$, the group of all $*$-automorphisms of
$\ga$. Systems of this type can be seen as the non-commutative
counterpart of topological dynamics and feature a much developed theory, for which
the reader is referred to {\it e.g.} \cite{S}.
In an effort to keep the present paper as self-consistent as we possibly can, however, we need to recall a few notions.
First, a state $\om$ on $\ga$ is said to be invariant if  $\om\circ\beta_h=\om$ for all $h\in H$.
On the GNS representation $(\ch_\om, \pi_\om, \xi_\om)$  of any such state $\om$ the action $\beta$ of $H$ can be implemented unitarily. More precisely, for each $h\in H$, $U_h^\om \pi_\om(a)\xi_\om:=\pi_\om(\beta_h(a))\xi_\om$, $a\in\ga$, defines a unitary on $\ch_\om$ such that $U_h^\om\pi_\om(a)U_{h^{-1}}^{\om}=\pi_\om(\beta_h(a))$ for all
$a\in\ga$. Lastly, a $C^*$-dynamical system $(\ga, H, \beta)$ is said to be $H$-abelian if, for any
$H$-invariant state $\om$, the family
$E_\om \pi_\om(\ga) E_\om$ is abelian, where $E_\om$ is the orthogonal projection
onto  $\ch_\om^H:=\{\xi\in\ch_\om: U_h^\om \xi=\xi\,, \textrm{for all}\, h\in H\}$.
$H$-abelianness is often inferred from asymptotical abelianness, see {\it e.g.} \cite[Proposition 3.1.16]{S}.
A dynamical system $(\ga, H, \beta)$ is called asymptotically abelian if there exists
a sequence $\{h_n: n\in\bn\}\subset H$ such that  for all $a, b\in\ga$ one has
$$\lim_{n\rightarrow\infty} \|  [\beta_{h_n}(a), b] \|=0\, ,$$
where $[\cdot, \cdot]$ denotes the commutator.\\
Finally, we recall that a state $\om$ invariant under the action of the shift $\tau$
is strongly clustering if $\lim_k\om(a\tau^{n_k}(b))=\om(a)\om(b)$ for all $a, b$ in $\ga$, for some diverging sequence
$\{n_k\}\subset\bn$, As is
known, if $\om$ is strongly clustering, then it is also extreme in the compact convex set
of all stationary states.

\begin{rem}
All of the present paper's results are stated for the CAR algebra indexed by $\bz$, but also hold true for
the CAR algebra indexed by $\bn$. Indeed, any stationary/spreadable/exchangeable state on
${\rm CAR}(\mathbb{N})$ uniquely extends to a stationary/spreadable/exchangeable (respectively) state on the bigger
$C^*$-algebra ${\rm CAR}(\mathbb{Z})$. 
We sketch below the easy proof.
Starting from a stationary/spredable/exchangeable state $\varphi $ on the smaller algebra, ${\rm CAR}(\mathbb{N})$,
one first defines a sequence of states $\{\varphi_n: n\in\mathbb{N}\}$, where, for each $n$, $\varphi_n$ is the state on
${\rm CAR}(\{k\in\mathbb{Z}: k\geq -n\})$ defined by $\varphi_n:=\varphi\circ\tau^n$ ($\tau$ being the shift on
${\rm CAR}(\mathbb{Z})$).
Now the sought stationary/spredable/exchangeable extension of $\varphi$ is the unique
state $\widetilde{\varphi}$ on ${\rm CAR}(\mathbb{Z})$ such that
the restriction of $\widetilde{\varphi}$ to ${\rm CAR}(\{k\in\mathbb{Z}: k\geq -n\})$ is $\varphi_n$ for every $n$.
Finally, uniqueness follows because if $\omega$ is another stationary/spredable/exchangeable extension of $\varphi$, then  the restriction of $\omega$ to ${\rm CAR}(\{k\in\mathbb{Z}: k\geq -n\})$ 
is $\varphi_n$ for every $n$.\\
\end{rem}

\section{Spreadable states}

Let $\om$ be a spreadable state on $\carf(\bz)$. Associated with any $h\in\bj_\bz$ there is an isometry acting on the Hilbert space
$\ch_\om$ of the GNS representation of $\om$ as
$$T_h^\om\pi_\om(a)\xi_\om:=\pi_\om(\a_h(a))\xi_\om\,, \,\, a\in\carf(\bz)\,.$$
We denote by $E_\om$ the orthogonal projection onto the closed subspace $\ch_\om^{\bj_\bz}:=\{\xi\in\ch_\om: T_h^\om\xi=\xi, \,\textrm{for all}\,\, h\in\bj_\bz \}$. Notice that $\bc\xi_\om\subset\ch_\om^{\bj_\bz}$, where
$\xi_\om$ is the GNS vector of $\om$.\\
Let $\bz\big[\frac{1}{2}\big]$ be the  (additive) group of dyadic numbers. For every fixed $n\in\bn$, we denote
by $\frac{\bz}{2^n}$ the set of rational numbers $\{\frac{k}{2^n}: k\in\bz\}$.\\
In order to perform the geometric construction alluded to in the introduction, we need to think of the generators of
$\bj_\bz$ as functions acting on the whole set of dyadic numbers. We do so by suitably extending $\theta_0$ to
a bijection $\widetilde{\theta}_0$ of $\bz\big[\frac{1}{2}\big]$, which is defined below

$$
\widetilde{\theta}_0(d):=\left\{\begin{array}{ll}
                      d & \text{if}\,\, d\leq -1\,, \\
                      2d+1&\text{if}\,\, -1\leq d\leq 0\\
                      d+1 & \text{if}\,\, d\geq 0\,.
                    \end{array}
                    \right.
$$
%Let $G_1$ be the subgroup of the group of all bijections of  $\bz\big[\frac{1}{2}\big]$ generated by
% $\widetilde{\theta}_0$ and $\widetilde{\tau}$, where
%$\widetilde{\tau}$ is the obvious extension
%of $\tau$ to $\bz\big[\frac{1}{2}\big]$, that is
%$\widetilde{\tau}(d)=d+1$ for all $d\in\bz\big[\frac{1}{2}\big]$.\\
 For each natural $n$, we consider the dilation $\delta_n$ by $2^n$ acting on $\bz\big[\frac{1}{2}\big]$, that is
$$\delta_n(d)=2^n d\,, \,\, d\in \bz\bigg[\frac{1}{2}\bigg]\,. $$
Let us define  $\widetilde{\theta_n}:=\delta_n^{-1}\circ\widetilde{\theta}_0\circ\delta_n$ and
$\widetilde{\tau}_{k, n}(r)=r+\frac{k}{2^n}\,, \, r\in \bz\big[\frac{1}{2}\big]$, for $n\in\bn$ and $k\in\bz$.\\
For each natural $n$, we then consider the group $G_n$ generated by
$\widetilde{\theta_k}$ and $\widetilde{\tau}_{1, n}$ with $k=1, \ldots, n$.
Note that $G_n\subset G_{n+1}$ for each $n$. Therefore, we can consider the group $G$ which is by definition the inductive  limit
of the sequence $\{G_n: n\in\bn\}$ w.r.t. the inclusions, namely
$$G:= \bigcup_{n} G_n\, .$$
%\begin{rem}
%It is worth stressing that  each $G_n$ is (isomorphic with) the Thompson group $F$
%as follows from {\it e.g.} Theorem 1.4.1 in \cite{JB}.
%\end{rem}
Notably, $G$  acts through automorphisms on $\carf\big(\bz\big[\frac{1}{2}\big]\big)$ by moving the corresponding indices, namely for every $g\in G$ a $*$-automorphism $\alpha_g$ of $\carf\big(\bz\big[\frac{1}{2}\big]\big)$ is uniquely defined by
$$\alpha_g(a_d):=a_{g(d)} \,\,\, d\in\bz\left[\frac{1}{2}\right]\,. $$
As usual, we denote by $\cs^G\big(\carf\big(\bz\big[\frac{1}{2}\big]\big)\big)$ the set of all invariant states on
$\carf\big(\bz\big[\frac{1}{2}\big]\big)$ under the action of $G$.\\
The next proposition allows us to realize any spreadable state on $\carf(\bz)$ as a unique
$G$-invariant state on $\carf\big(\bz\big[\frac{1}{2}\big]\big)$.
\begin{prop}
\label{affine}
The map $T:\cs^G\big(\carf\big(\bz\big[\frac{1}{2}\big]\big)\big)\rightarrow \cs^{\bj_\bz}(\carf(\bz))$
$$T(\om):=\om\lceil_{\carf({\bz})},\,\, \om\in\cs^G\left(\carf\left(\bz\left[\frac{1}{2}\right]\right)\right) $$
establishes an affine homeomorphism of compact convex sets.
\end{prop}

\begin{proof}
We start by observing that the map is well defined, in that it sends a $G$-invariant state on
$\carf\big(\bz\big[\frac{1}{2}\big]\big)$ to a spreadable state on $\carf({\bz})$.
The statement amounts to proving that, conversely, any spreadable state on $\carf({\bz})$ can uniquely be lifted to a $G$-invariant state on $\carf\big(\bz\big[\frac{1}{2}\big]\big)$. To do so, for each $n\geq 1$, we consider the
$C^*$-algebra $\ga_n:=\carf\big(\frac{\bz}{2^n}\big)$. Since $\frac{\bz}{2^n}\subset\frac{\bz}{2^{n+1}}$, we also have $\ga_n\subset\ga_{n+1}$. Moreover, the map $\Psi_n: \ga_{n+1}\rightarrow\ga_n$ given by
$\Psi_n(a_r):=a_{2r}$ for every $r\in\frac{\bz}{2^{n+1}}$ is actually
a $*$-isomorphism by universality of the $\carf$ algebras.\\
Let now $\varphi_0$ be any spreadable state on $\carf(\bz)$.
We first extend $\varphi_0$ to $\ga_n$ inductively by setting
 $\varphi_{n+1}:=\varphi_n\circ\Psi_n$, for $n\geq 0$.
By definition, each $\varphi_n$ is a spreadable state on $\carf\big(\frac{\bz}{2^n}\big)$, by which we mean
that $\varphi_n$ is invariant under the natural action of all strictly monotone maps on $\frac{\bz}{2^n}$ with
cofinite range. Now $\bigcup_{n=1}^\infty\ga_n$ is dense in $\carf\big(\bz\big[\frac{1}{2}\big]\big)$, which means
a state $\varphi$ can be defined on the whole $\carf\big(\bz\big[\frac{1}{2}\big]\big)$ as inductive limit of
the sequence of states $\{\varphi_n: n\in\bn\}$.
It only remains to ascertain that the state thus obtained is $G$-invariant. To this end, it is
enough to note that, for each $n$, the state $\varphi_n$ is invariant under the action of $\widetilde{\theta_k}$ and $\widetilde{\tau}_{l, n}$,
 for $k=1, \ldots, n$ and $l$ in $\bz$.
At this point, $G$-invariance of $\varphi$  follows immediately by density of $\bigcup_n \ga_n$ in 
 $\carf\big(\bz\big[\frac{1}{2}\big]\big)$  and the fact that $G$ is generated by $\widetilde{\theta_k}$ and $\widetilde{\tau}_{l, n}$,
as $n$ varies in $\bn$ and $k, l$ vary in $\bz$.\\
The argument above also shows that the extension obtained is in fact
the only $G$-invariant extension, because $\varphi_n$ is the only spreadable extension
of $\varphi_0$ to $\ga_n$.\\
The continuity of the restriction map is entirely obvious. Finally, its inverse is continuous as well
by compactness.
\end{proof}
For any given $G$-invariant state $\om$ on $\carf(\bz\big[\frac{1}{2}\big])$ and for each $g\in G$, we can define
a unitary $T_g^\om$ acting on $\ch_\om$ as
$$T_g^\om\pi_\om(a)\xi_\om=\pi_\om(\alpha_g(a))\xi_\om\,,\, \, a\in\carf\bigg(\bz\left[\frac{1}{2}\right]\bigg)\, .$$
We then consider
the closed subspace
$$\ch_\om^G:=\{\xi\in\ch_\om: T_g^\om\xi=\xi,\, \textrm{for all}\, g\in G\}$$
of invariant
vectors, and denote by $E^G_\om$ the corresponding orthogonal projection.
Finally, we denote by $\om_0$ the restriction of $\om$ to $\carf(\bz)$.\\
We next need to consider the unital semigroup $S\subset G$ generated by
$\widetilde{\theta_n}$ and $\widetilde{\tau}_{k, n}$ as $n$ varies in $\bz$ and $k$ in $\bn$. Note that for any
element $h\in S$ one has $h(q)\geq q$ for all $q$ in $\bz\big[\frac{1}{2}\big]$.\\
Now, for any $G$-invariant state $\om$ on $\carf{(\bz\big[\frac{1}{2}\big])}$ define
$\ch_\om^S:=\{\xi\in\ch_\om: T_h^\om\xi=\xi,\, h\in S\}$.

\begin{lem}\label{fromGtoS}
For any $G$-invariant state $\om$ on $\carf{(\bz\big[\frac{1}{2}\big])}$, one has  then $\ch_\om^G=\ch_\om^S$.
\end{lem}

\begin{proof}
The inclusion $\ch_\om^G \subseteq \ch_\om^S$ is trivial since $S\subset G$.
As for the  inclusion $ \ch_\om^S\subseteq \ch_\om^G$, let $\xi$ be a vector sitting in $\ch_\om^S$.
In order to prove that $\xi$ lies in $\ch_\om^G$, it is enough to observe that $S$ generates $G$ as a group and
that a vector that is invariant for a unitary is also invariant for its adjoint.
\end{proof}
As an application of the above lemma, we find that the projections  $E_\om^G$ and $E_\om^S$  onto 
 $\ch_\om^G$ and $\ch_\om^S$, respectively, are actually the same. 
\begin{lem}\label{Gabelianness}
The dynamical system $\left(\carf\big(\bz\big[\frac{1}{2}\big]\big), G, \alpha\right)$ is  $G$-abelian.
\end{lem}

\begin{proof}
Set $\ga:=\carf\big(\bz\big[\frac{1}{2}\big]\big)$ for brevity.  We have to prove that for any
$G$-invariant state $\om$, the operators of the set $E_\om^G\pi_\om(\ga)E_\om^G$ commute with one another, where
$E_\om^G$ is the projection onto the subspace $\ch_\om^G=\{\xi\in\ch_\om: T^\om_g\xi=\xi, \, 	\textrm{for all}\, g\in G\}$ of all invariant vectors.\\
We shall accomplish this task in two steps.  The first step is to prove that
 $E_\om^G\pi_\om(a)E_\om^G=0$ for any odd element
$a\in\ga$. The second step is to show that
$E_\om^G\pi_\om(\ga_+)E_\om^G$ is abelian, where $\ga_+$ is the even subalgebra
of $\ga$. \\
In order to take  the first step, we closely follow the argument in Example 5.2.21 in \cite{BR2}.
Now it is straightforward to verify that for all odd $a$ in $\ga$, we have
$$\lim_{n\rightarrow\infty} \|\{\alpha_{g_n}(a), a^*\}\|=0\, ,$$
with $g_n:=\widetilde{\tau}_{n, 0}$, that is $g_n(r)=r+n$ for every $r\in\bz\big[\frac{1}{2}\big]$.
Now denote by $F_\om$ the orthogonal projection onto the closed subspace $\{\xi\in\ch_\om: T^\om_{g_n}\xi=\xi\,, \textrm{for all}\, n\in\bn\}$. Note that $E_\om^G\leq F_\om$. We aim to show that
$F_\om\pi_\om(a) F_\om=0$ if $a$ is odd. By applying  the von Neumann ergodic theorem to the projection
$F_\om$, one sees that for any such $a$ it holds
$$
\{F_\om\pi_\om(a )F_\om, F_\om\pi_\om(a^*)F_\om\}=\lim_{n\rightarrow\infty} \frac{1}{n+1}\sum_{k=0}^n
F_\om \pi_\om(\{\a_{g_k}(a), a^*\})F_\om=0\,.
$$
But then $F_\om\pi_\om(a )F_\om=0$, hence $$E_\om^G\pi_\om(a) E_\om^G=E_\om^GF_\om\pi_\om(a) F_\om E_\om^G=0\, .$$
Note that in particular any $G$-invariant state $\om$ is even.\\
As for the second step, we start by observing that $\ga_+$ is invariant under the action of $G$ as 
$\alpha_g\circ\theta=\theta\circ\alpha_g$ for all $g\in G$, where $\theta$ is the grading 
of $\carf\big(\bz\big[\frac{1}{2}\big]\big)$.  The dynamical system $(\ga_+, \alpha, G)$ obtained by restricting the dynamics to $\ga_+$ is asymptotically abelian. Indeed, if we consider again the sequence $g_n:=\widetilde{\tau}_{n, 0}$, that is $g_n(r)=r+n$ for every $r\in\bz\big[\frac{1}{2}\big]$, one can verify that for all localized $a, b\in\ga_+$ one has
$[\alpha_{g_n}(a), b]=0$ for $n$ big enough. But then from a standard density argument it follows that
$$\lim_{n\rightarrow\infty} \|[\alpha_{g_n}(a), b]\|=0\, ,$$
for all $a, b\in\ga_+$.
By applying Proposition 3.1.16 in \cite{S} we finally see that $(\ga_+, \alpha, G)$ is $G$-abelian.\\
We are now ready to reach the conclusion. To this end, define $\ch_{\om, \pm}:=\overline{\pi_\om(\ga_\pm)\xi_\om}$. Since $\om$ is even, it follows that $\ch_\om$ decomposes as $\ch_\om=\ch_{\om, +}\oplus\ch_{\om, -}$.
We then consider the orthogonal projections, $E^G_{\om, \pm}$, onto $\ch_{\om, \pm}\cap\ch_\om^G$, respectively. We now show that $E^G_{\om, -}$ is zero. This can be seen by noting that it projects onto the closure of $\{\pi_\om(b)\xi_\om: b\,\, \textrm{odd and G-invariant}\}$.  But thanks to the first step one sees that this subspace is zero.
Therefore, we have  
$$E_\om^G\pi_\om(\ga)E_\om^G=E_\om^G\pi_\om(\ga_+)E_\om^G=E^G_{\om, +}\pi_\om(\ga_+) E^G_{\om, +}\,,$$ and the latter set is commutative, as we saw in the second step.

%In order to take  the second step, we closely follow the argument in Example 5.2.21 in \cite{BR2}.
%As above, it is straightforward to verify that for all odd $a$ in $\ga$, we have
%$$\lim_{n\rightarrow\infty} \|\{\alpha_{g_n}(a), a^*\}\|=0\, .$$
%Now denote by $F_\om$ the orthogonal projection onto the closed subspace $\{\xi\in\ch_\om: T^\om_{g_n}\xi=\xi\,, \textrm{for all}\, n\in\bn\}$. Note that $E_\om^G\leq F_\om$. We aim to show that
%$F_\om\pi_\om(a) F_\om=0$ if $a$ is odd. By applying  the von Neumann ergodic theorem to the projection
%$F_\om$, one sees that for any such $a$ it holds
%$$
%\{F_\om\pi_\om(a )F_\om, F_\om\pi_\om(a^*)F_\om\}=\lim_{n\rightarrow\infty} \frac{1}{n+1}\sum_{k=0}^n
%F_\om \pi_\om(\{\a_{g_k}(a), a^*\})F_\om=0\,.
%$$
%But then $F_\om\pi_\om(a )F_\om=0$, hence $$E_\om^G\pi_\om(a) E_\om^G=E_\om^GF_\om\pi_\om(a) F_\om E_\om^G=0\, .$$
\end{proof}
As an application of the previous result, we find that $E_\om^S$ is a rank-one projection if $\om$ is an extreme $G$-invariant state, as follows from Proposition 3.1.12 in \cite{S}.
We are now ready to state and prove our main result.

\begin{thm}\label{main}
On $\carf(\bz)$ any spreadable state is also exchangeable, that is
$$\cs^{\bj_\bz}(\carf(\bz))=\cs^{\bp_\bz}(\carf(\bz))\, .$$
\end{thm}

\begin{proof}
We need only prove the inclusion $\cs^{\bj_\bz}(\carf(\bz))\subseteq \cs^{\bp_\bz}(\carf(\bz))\, $, for the reverse inclusion
is trivial. By the Krein-Milman theorem, it is enough to handle extreme spreadable state.
Precisely, our strategy is to prove that  any extreme spreadable state $\om$ is an infinite product of  a single even state
$\rho$ on $\bm_2(\bc)$. \\
In the light of Proposition \ref{affine}, we are reconducted to showing this:
the restriction of any extreme
state $\om$ in $\cs^G\big(\carf\big(\bz\big[\frac{1}{2}\big]\big)\big)$  to $\carf(\bz)$ is an infinite product of  a single even state
$\rho$ on $\bm_2(\bc)$. \\
Now, if $\om$ is an extreme $G$-invariant state on $\carf\big(\bz\big[\frac{1}{2}\big]\big)$, then
$E_\om^G$ is the rank-one projection onto $\bc\xi_\om$ thanks Lemma \ref{Gabelianness} and 
Proposition 3.1.12 in \cite{S}. By virtue of Lemma \ref{fromGtoS}, $E_\om^S$ is the projection onto
$\bc_{\xi_\om}$ as well.\\
The next thing to do is to apply the Alaloglu-Birkhoff theorem, see \emph{e.g.} Proposition 4.3.4 in \cite{BR1}, to the family  of unitaries
$\{T_h^\om: h\in S\}\subseteq\cb(\ch_\om)$. This
yields a net $\{S_\gamma: \gamma\in I\}$, whose terms are finite convex combinations of the form $S_\gamma=\sum_{k=1}^{n_\gamma} \lambda_{k}^\gamma  T_{h_k^\gamma}^\om$, for some $\lambda_k^\gamma\geq 0$ with $\sum_{k=1}^{n_\gamma} \lambda_k^\gamma=1$, which is
strongly convergent to  $E_\om^S$. 
%Now, by Lemma \ref{fromGtoS}  we have that $E_\om^S$ is the rank-one projection $E_{\bc\xi_\om}$.\\
Let now $a, b$ be fixed elements of $\carf\big(\bz\big[\frac{1}{2}\big]\big)$. For each $\gamma$ in $I$, define $b_\gamma:=\sum_{k=1}^{n_\gamma} \lambda_{k}^\gamma  \a_{h_k^\gamma}(b)$. We have
\begin{align}
\begin{split}
\label{clustering}
\lim_\g \om(ab_\g)&= \lim_\g \langle \pi_\om(a b_\g) \xi_\om, \xi_\om \rangle\\
&= \lim_\g  \langle  \pi_\om(a)  \sum_{k=1}^{n_\g}\lambda_{k}^\g T_{h_k^\g}^\om \pi_\om(b)\xi_\om, \xi_\om\rangle
\\
&=\langle  \pi_\om(a)   E_{\bc\xi_\om}\pi_\om(b)\xi_\om, \xi_\om\rangle=
\om(a)\om(b)\,.
\end{split} .
\end{align}
We are now ready to prove our claim. Let $\rho$ be the even state on $\bm_2(\bc)$ defined by
$\rho(A):=\om(i_l(A))$, $A$ in $\bm_2(\bc)$, where $i_l$ is any of the embeddings
of $\bm_2(\bc)$ into its infinite $\bz_2$-graded tensor product. Note that $\rho$ is well
defined, {\it i.e.} the definition does not depend on $l$, because $\om$ is spreadable.
In order to show that $\om_0$ (that is the restriction of $\om$ to $\carf(\bz)\,$) coincides with the infinite product of $\rho$ with itself, we need to prove
the equality
$$\om_0(i_{j_1}(A_1)i_{j_2}(A_2)\cdots i_{j_n}(A_n))=\rho(A_1)\rho(A_2)\cdots\rho(A_n)$$
for every $n\in\bn$, for every $j_1< j_2< \ldots< j_n\in\bz$, and for every $A_1, A_2, \ldots, A_n$ in $\bm_2(\bc)$.
This task can be accomplished by induction on $n$. For $n=1$, there is nothing to prove.
Let us move on to take the inductive step. We will show that the equality holds with $n+1$ if
it holds with $n$. To this end, we start by observing that, for all $n\in\bn$, $j_1<j_2<\ldots<j_{n+1}$ in $\bz$, and $A_1, \cdots, A_n, B\in\bm_2(\bc)$, $G$-invariance of $\om$
gives
$$\om_0(i_{j_1}(A_1)\cdots i_{j_n}(A_n)i_{j_{n+1}}(B))=\om(i_{j_1}(A_1)\ldots i_{j_n}(A_n) i_{h_k^\g(j_{n+1})}(B))\,,$$
where, for each $\g$ in $I$, the $h_k^\g$'s are the monotone functions in $S$ (in particular, $h_k^\g(m)\geq m$ for all $m$ in $\bz$) appearing in the definition of the net $S_\g$ we introduced above. By summing the equalities above on all $k$'s between $1$ and $n_\g$, we find
\begin{align*}
&\om_0(i_{j_1}(A_1)\cdots i_{ j_n}(A_n)i_{j_{n+1}}(B))\\
=&\sum_{k=1}^{n_\g}\lambda_k^\g\om(i_{j_1}(A_1)\cdots i_{j_n}(A_n) i_{h_k^\g(j_{n+1})}(B))\\=&\om(i_{j_1}(A_1)\cdots i_{j_n}(A_n)b_\g)
\end{align*}

\medskip
where $b_\g:=\sum_{k=1}^{n_\g}\lambda_k^\g i_{h_k^\g(j_{n+1})}(B)=\sum_{k=1}^{n_\g} \lambda_{k}^\g \a_{h_k^\g}(i_{j_{n+1}}(B))$.\\

\medskip
Now thanks to Equalities \ref{clustering} (applied to $b=i_{j_{n+1}}(B)$) we have

$$\lim_\g\om(i_{j_1}(A_1)\cdots i_{j_n}(A_n)b_\g)=\om_0(i_{j_1}(A_1)\cdots i_{j_n}(A_n))\rho(B)\,,$$

and we are done because by our inductive hypothesis we have

$$\om_0(i_{j_1}(A_1)\cdots i_{j_n}(A_n))\r(B)=\rho(A_1)\cdots\rho(A_n)\rho(B)\, .$$
\end{proof}

Furthermore, exchangeable (or spreadable) states make up a face of the larger convex
of all stationary states, as we next show.

\begin{prop}\label{face}
The convex subset $\cs^{\bp_\bz}(\carf(\bz))=\cs^{\bj_\bz}(\carf(\bz))$ of exchangeable/spreadable states of $\carf(\bz)$ is a proper face of  the convex set of all stationary states $\cs^{\bz}(\carf(\bz))$.
\end{prop}

\begin{proof}
We start by showing that the extreme states of $\cs^{\bp_\bz}(\carf(\bz))$ are extreme
in $\cs^{\bz}(\carf(\bz))$ as well. This  can be seen by showing that such states are strongly clustering
(w.r.t the action of $\tau$).
Now any such state $\om$, is by \cite[Theorem 5.3]{CFCMP}  the infinite product of a single even
state $\r$ on $\bm_2(\bc)$.
If $a, b$ in $\carf(\bz))$ are localized, one has $\om(a\tau^n(b))=\om(a)\om(b)$ eventually for $n$ big enough. The full conclusion follows by a standard density argument since localized elements are dense. \\
We are now ready to reach the conclusion. Let $\varphi$ be a symmetric state which is a convex combination of the type
$\varphi=t\varphi_1+(1-t)\varphi_2$, $t\in(0, 1)$, where $\varphi_1$ and $\varphi_2$ are stationary states.
We need to make sure that both $\varphi_1$ and $\varphi_2$ are actually symmetric. Denote by $\mathcal{K}\subset\cs^{\bp_\bz}(\carf(\bz))$ the compact set of extreme symmetric states.
Since $\cs^{\bp_\bz}(\carf(\bz))$ is a Choquet simplex, there exists a unique measure
$\mu$, which is supported on $\mathcal{K}$, such that $\varphi=\int_K \varphi_\mu{\rm d}\nu(\mu)$.
Since  $\cs^\bz(\carf(\bz))$ is a Choquet simplex as well, see Example 5.2.21 in \cite{BR2},
the decomposition above is also the (unique) decomposition
 of $\varphi$ as a point of $\cs^\bz(\carf(\bz))$. Denote now by $\mathcal{E}$ the set of all extreme stationary states.
Again, since $\cs^\bz(\carf(\bz))$ is a Choquet simplex, there exist unique measures $\nu_1, \nu_2$ supported\footnote{By this we mean that $\nu_i(\mathcal{E})=1$, $i=1, 2$, and this equality makes perfect sense as
$E$ is a $G_\delta$-set thanks to the general lemma of Choquet, see {\it e.g.} \cite[3.4.1]{S}.} on $\mathcal{E}$ such that
$$\varphi_i=\int_{\mathcal{E}} \om{\rm d}\nu_i(\om)\,, i=1, 2.$$
By uniqueness of the decomposition, the equality $\nu=t\nu_1+(1-t)\nu_2$ must hold. In particular, the support of
$ t\nu_1+(1-t)\nu_2$ is the same as the support of $\nu$, which is contained in $\mathcal{K}$. But then
the supports of both $\nu_1$ and $\nu_2$ are contained in $\mathcal{K}$ as well, which means
$\varphi_1$ and $\varphi_2$ are symmetric.
\end{proof}

\begin{rem}
Despite $\cs^{\bp_\bz}(\carf(\bz))$ being a face of $\cs^\bz(\carf(\bz)$, the so-called Olshen
theorem \cite{O} fails on the CAR algebra. More explicitly, there exist exchangeable states on $\carf(\bz)$
whose exchangeable algebra is strictly contained in the stationary algebra.
The vacuum state itself provides the most elementary counterexample. 
Indeed, its exchangeable algebra $\pi_{\om_\Om}''(\carf(\bz))_{\bp_\bz}$
is trivial as follows from \cite[Proposition 2.9, Proposition 2.10]{CRZbis}.
However, its stationary algebra $\pi_{\om_\Om}''(\carf(\bz))_{\bz}=\{U_{\om_\Om}^\tau\}'$ is not trivial.

\end{rem}

Theorem \ref{main} allows us to determine the structure of the set of all spreadable states on the so-called
self-adjoint subalgebra $\mathfrak{C}(\bz)$ of $\carf(\bz)$. This is by definition the $C^*$-subalgebra generated by
the set $\{x_i: i\in\bz\}$, with $x_i:= a_i+ a_i^\dagger$, $i\in\bz$.
It turns out that this set only consists of the vacuum state.

\begin{cor}\label{selfadjoint}
The vacuum state is the only spreadable state on $\mathfrak{C}(\bz)$.
\end{cor}

\begin{proof}
Let $\varphi$ be a spreadable state on $\mathfrak{C}(\bz)$. Consider any extension $\widetilde{\varphi}$ of $\varphi$ to
$\carf(\bz)$. For each $n\in\bn$, define $\widetilde{\varphi}_n:=\frac{1}{|F_n|}\sum_{h_\in F_n}\widetilde{\varphi}\circ \a_h$,
where $\{F_n: n\in\bn\}$ is the  F{\o}lner  sequence of $\bj_\bz$. By weak$^*$ compacteness, up to extracting a subsequence, we can suppose that the sequence $\widetilde{\varphi}_n$
weakly$^*$ converges to a state $\om$. By construction, $\om$ is a spreadable state on $\carf(\bz)$ and
its restriction to $\mathfrak{C}(\bz)$ is the same as $\varphi$. Now by Theorem \ref{main} $\om$ is exchangeable, and thus $\varphi$ is exchangeable as well. The thesis is finally reached as an application of
Proposition 4.3 in \cite{CRcar}, where the vacuum state is shown to be the unique exchangeable state on
$\mathfrak{C}(\bz)$.
\end{proof}

Before ending the section, we would like to point out that our techniques can also be made use of  to address 
infinite tensor products of a (nuclear) sample $C^*$-algebra $\ga$. 
Denote by $\bigotimes_\bz \ga$ the infinite tensor product of $\ga$ with itself.
Both $\bp_\bz$ and $\bj_\bz$ act  naturally on $\bigotimes_\bz \ga$.
Invariant states under the action of permutations, often referred to as symmetric states, make up a Choquet simplex whose extreme points are infinite product of a single state on $\ga$, see Theorem 2.7 in \cite{Sto}.\\
%Now the action of $\bj_\bz$ on  $\bigotimes_\bz \ga$ is easily seen to be asymptotic abelian. 
As with the CAR algebra,
the action of $\bj_\bz$ on  $\bigotimes_\bz \ga$  can be extended to an action of $G$ on $\bigotimes_{\bz\left[\frac{1}{2}\right]} \ga \supseteq\bigotimes_\bz \ga $. Again, the $C^*$-dynamical system thus obtained is
$G$-abelian. As a consequence, the projection $E_\om$ associated with any extreme spreadable state $\om$ is
one-dimensional. This allows one to  exploit the arguments employed in the proof of Theorem \ref{main}. Therefore, one has that the Choquet simplex of spreadable states on $\bigotimes_\bz \ga$  is the same as the simplex of symmetric states.
This applies in particular to commutative sample $C^*$-algebras. In this way one also finds an 
independent proof of the classical Ryll-Nardzweski theorem (for bounded random variables).

\section{Rotatable states}

In this section we show that  the set of rotatable states on the CAR algebra agrees with the set of exchangeable states.
This result can in a sense be regarded as a version of Freedman's theorem for the CAR algebra, in that it can be combined
with Theorem 5.3 and Theorem 5.5 in \cite{CFCMP} to provide
an explicit description of what rotatable states look like.

\begin{thm}\label{rot}
Let $\om$ be a state on $\carf(\bz)$. The following are equivalent:
\begin{itemize}
\item [(i)] $\om$ is exchangeable;
\item [(ii)] $\om$ is spreadable;
\item [(iii)] $\om$ is rotatable.
\end{itemize}

\end{thm}

\begin{proof}
The equivalence between (i) and (ii) has been established in Theorem \ref{main}.
As for the equivalence between (i) and (iii), we need only prove $\mathcal{S}^{\bp_{\bz}}({\rm CAR}(\bz))\subseteq\mathcal{S}^{\mathbb{O}_{\bz}}({\rm CAR}(\bz))$.
To this end,
it is enough to ascertain that the extreme states in $\mathcal{S}^{\bp_{\bz}}({\rm CAR}(\bz))$ are rotatable thanks to
the Krein-Milman theorem as  $\mathcal{S}^{\bp_{\bz}}({\rm CAR}(\bz))$ and $\mathcal{S}^{\mathbb{O}_{\bz}}({\rm CAR}(\bz))$
are convex and weakly$^*$ compact.\\
As proved in \cite[Theorem 5.3 ]{CFCMP}, the extreme states of $\mathcal{S}^{\bp_{\bz}}({\rm CAR}(\bz))$ are precisely the Araki-Moriya product states $\varphi_\mu$, $\mu\in [0,1]$, with
$\varphi_\mu=\times_{\bz}\, \rho_\mu$ and $\rho_\mu$ is the state on $\bm_2(\bc)$ given by
\begin{equation}\label{rhostate}
\r_\mu \left(
\begin{array}{ll}
\alpha & \beta \\
\gamma & \delta
\end{array}
\right) =\mu \alpha+\left( 1-\mu \right) \delta\,.
\end{equation}
Note that in particular $\varphi_\mu(a_i^\dagger a_i)=1-\mu$ for all $i\in\bz$.\\
Now these product states are known to be gauge-invariant quasi-free states with covariance operator $(1-\mu)I$, namely
\begin{equation}\label{quasifree}
\varphi_\mu(a^*(f_n)	\ldots a^*(f_1)a(g_1)\ldots a(g_m))= \delta_{n, m} {\rm det} [\langle (1-\mu)f_i, g_j \rangle]
\end{equation}
for all $n, m\in\bn$, and $f_1, \ldots, f_n, g_1, \ldots g_m\in \ell_2(\bz)$, where $\langle\cdot, \cdot\rangle$ denotes the standard inner product  of $\ell_2(\bz)$, which is taken anti-linear to the left and linear to the right.\\
In order to prove \eqref{quasifree}, we first verify it for basis vectors, in which case we find
$\varphi_u(a^*_{i_1}\ldots a^*_{i_n}a_{j_1}\ldots a_{j_m})=\delta_{n, m}{\rm det}[(1-\mu)\delta_{i_l, j_k}]$.
The determinant in the right-hand side of the last equality is $0$ if $\{i_1, \ldots, i_n\}\neq \{j_1, \ldots, j_m\}$
(that is when there is at least one creator $a^*_{i_l}$ which cannot be matched withe the corresponding annihilator
$a_{i_l}$). When creators and annihilators match, there are two cases to analyze. In the first, the indices appearing in the monomial
$a^*_{i_1}\ldots a^*_{i_n}a_{j_1}\ldots a_{j_m}$  are ordered ($i_1<i_2\ldots <i_n$ and $j_1>j_2>\ldots> j_n$), in which case both sides equal $(1-\mu)^n$. In the second, the indices are not ordered. Re-ordering the indices entails
exchanging two rows or columns of the matrix on the right of the equality, and exchanging two next operators on the left of the equality, and both operations make a negative sign appear each time.\\
By linearity of $a(f)$ w.r.t. to its argument $f$ and by linearity of the determinant w.r.t. taking linear combinations of a single
row/column at a time, Equality \eqref{quasifree} holds when  $f_i$'s and $g_j$'s are finite linear combinations
of basis vectors. Finally, the full conclusion is reached by density and continuity.\\
Now invariance of $\varphi_\mu$ under all Bogolubov automorphisms $\rho_O$ is a straightforward
consequence of \eqref{quasifree}.

\end{proof}

%\begin{cor}\label{even}
%Rotatable, spreadable, and exchangeable states on the CAR algebra are even
%\end{cor}
%
%\begin{proof}
%It is a straightforward consequence of Theorem \ref{rot} because exchangeable states are certainly even.
%\end{proof}

\section*{Acknowledgments}
\noindent
All authors acknowledge  the support of the Italian INDAM-GNAMPA Project Code CUP\_E55F22333270001, 
the Italian PNRR MUR project PE0000023-NQSTI, and Centro Nazionale CN00000013 CUP H93C22000450007.\\

\noindent
\textbf{Data Availibility} Data sharing not applicable to this article as no datasets were generated or analyzed during
the current study. The authors are not aware of a conflict of interest on their side related to this article.

\end{document}